\documentclass[12pt,reqno,a4paper]{amsart}
\usepackage[margin = 1in]{geometry}
\usepackage{amsfonts,amsmath,times,amsthm,amssymb}
\usepackage{tikz}
\usepackage{cool} %Package for generalized hypergeometric functions
\usepackage{hyperref}
\usepackage{url}
\usepackage{soul}

\newtheorem{theorem}{Theorem}
\newtheorem{lemma}{Lemma}

\newtheorem{prop}{Proposition}
\newtheorem{corollary}{Corollary}

\newcommand{\indicator}{{\bf 1}}
\newcommand{\field}{\mathbb{F}}
\renewcommand{\E}{\mathbb{E}}
\newcommand{\given}{\, | \,}
\newcommand{\convL}{\, \overset{L_1}{\longrightarrow} \,}
\newcommand{\Var}{\mathbb{V}{\rm ar}}
\newcommand{\convLL}{\, \overset{L_2}{\longrightarrow} \,}
\newcommand{\Given}{\, \Bigg{|} \,}
\newcommand{\convP}{\, \overset{P}{\longrightarrow} \,}
\newcommand{\ggiven}{\, \Big{|} \,}
\newcommand{\convD}{\, \overset{D}{\longrightarrow} \,}
\newcommand{\convas}{\, \overset{a.s.}{\longrightarrow} \,}

\begin{document}
\begin{center}
	{\Large \bf  
		Several topological indices of random caterpillars}
	
	\bigskip
	{\bf Panpan Zhang\footnote{\, Corresponding author; Email: 
	panpan.zhang@pennmedicine.upenn.edu.}}
	
	\medskip
	{\tiny Department of Biostatistics, Epidemiology and Informatics, Perelman School of Medicine, University of Pennsylvania, Philadelphia, PA 19104, U.S.A.}
	
	\medskip
	{\bf Xiaojing Wang\footnote{\, Email: xiaojing.wang@uconn.edu.}}
	
	\medskip
	{\tiny Department of Statistics, University of Connecticut, Storrs, CT 06269, U.S.A.}
	
	\bigskip
	
	\today
\end{center}

\bigskip\noindent
{\bf Abstract.}  In chemical graph theory, caterpillar trees have been an appealing model to represent the molecular structures of benzenoid hydrocarbon. Meanwhile, topological index has been thought of as a powerful tool for modeling quantitative structure-property relationship and quantitative structure-activity between molecules in chemical compounds. In this article, we consider a class of caterpillar trees that are incorporated with randomness, called random caterpillars, and investigate several popular topological indices of this random class, including Zagreb index, Randi\'{c} index and Wiener index, etc. Especially, a central limit theorem is developed for the asymptotic distribution of the Zagreb index of random caterpillars.


Secondary: 50C10, 60F05


\section{Introduction}
\label{Sec:Intro}

In this article, we investigate several topological indices of a class of random graphs. The {\em topological index} of a graph is a graph-invariant descriptor that quantifies its structure or some kind of feature. Topological index has found a plethora of applications in chemical graph theory, mathematical chemistry and chemoinformatics. In practice, various kinds of topological indices, such as Zagreb index~\cite{Gutman1972} and Randi\'{c} index~\cite{Randic}, are used to compare molecular graphs of chemical compounds~\cite{Todeschini}, and to model quantitative structure-property relationship (QSPR) and quantitative structure-activity relationship (QSAR) between molecules~\cite{Devillers}. We refer interested readers to~\cite{Todeschini} for a text-style exposition of utilization of topological index in chemistry.

Specifically, we look into a class of random graphs that incorporate randomness into {\em caterpillar graphs}, i.e., {\em random caterpillars}. In mathematical chemistry and chemoinformatics, a caterpillar graph (or simply caterpillar) is an acyclic graph with the property that there remains a path (called spine) if all leaves are pruned, best known for modeling the structure and intrinsic properties of benzoid hydrocarbon molecules~\cite{Andrade, ElBasil, ElBasil1990, Gutman1985}.

Our motivation of incorporating randomness and caterpillar graphs is from a recent article~\cite{Kryven}, in which a random graph model was utilized to model the chemistry of a discrete polymerization process. More precisely, we consider random caterpillars that grow in a uniform manner. At time $0$, there is a {\em spine} consisting of $m \ge 2$ (fixed) nodes. At each subsequent point of time, a {\em leaf} is connected with one of the spine nodes by an edge, all spine nodes being equally likely to be selected. At time $n$, we denote the structure of a random caterpillar by $C_n$. 

At first, we present some known results of $C_n$ as preliminaries, and give some notations that will be used throughout the manuscript. At time $n$, there is a total of $(n + m)$ nodes. Additionally, the total number of edges is fixed, i.e., $(n + 2m - 2)$. Enumerate $m$ spine nodes in a preserved order (e.g., from left to right) with distinct numbers in $\{1, 2, \ldots, m\}$. Let $X_{i, n}$ be the number of leaves attached to spine node $i$ for $i = 1, 2, \ldots, m$, and let $D_{i, n}$ be the degree of spine node $i$. According to the evolution of random caterpillars, we know that the joint distribution of $X_{i, n}$'s is multinomial with parameters $n$ and $\mathbb{I}_{m \times 1}/m$, where $\mathbb{I}_{m \times 1}$ is a column vector of all $1$'s. Additionally, there is an instantaneous relation between $X_{i, n}$ and $D_{i, n}$. That is $D_{i, n} = X_{i, n} + 1$ for $i = 1, m$; $D_{i, n} = X_{i, n} + 2$ for $i = 2, 3, \ldots, m - 1$. 

In the remainder of the paper, we calculate several topological 
indices of $C_n$. They are Gini index, Hoover index, Zagreb index, 
Randi\'{c} index and Wiener index, respectively presented from 
Section~\ref{Sec:Gini} to Section~\ref{Sec:Wiener}. We give the 
definition of each index and some brief introductions about their 
applications that initiate and promote the motivations of our 
analysis in the sequel. In Section~\ref{Sec:numeric}, we carry out 
numerical experiments to verify the theoretical results developed in 
the preceding sections. Some concluding remarks are addressed at the 
end of the article .

\section{Gini index}
\label{Sec:Gini}

The {\em Gini index} (or called Gini coefficient) is a widely used measure in economics~\cite{Gini}, mainly known for assessing statistical dispersion of income or wealth of a population. More precisely, the Gini index (of a target population) is a number measuring the degree of inequality in (income or wealth) distribution. Given a population of size $n$, the {\em Gini index} is given by
$$G = \frac{\sum_{i = 1}^{n} \sum_{j = 1}^{n} |w_i - w_j|}{2n\sum_{i = 1}^{n} w_i},$$
where $w_i$ refers to the wealth of person $i$.

Recently, different types of Gini index that are well defined for graphs were proposed. We present the results only in this section without proof, as random caterpillars were used as examples in all the relevant sources. A distance-based Gini index for rooted trees was given in~\cite{Balaji}, used as a measure of disparities of trees within random tree classes. Let $B_n$ be the distance-based Gini index of $C_n$. The exact and the asymptotic mean of $B_n$ were calculated in~\cite{Balaji}, and revisited by~\cite{Zhang}.

\begin{prop}[\cite{Balaji}]
	The mean of the distance-based Gini index for a random caterpillar is
	$$\E\left[G_n\right] = \frac{(2m^2 - 2)n^2 + (m^3 + 4m^2 - m + 2)n + 2m^4 - 2m^2}{(6m^2 + 6m)n^2 + 12m^3n + 6m^4 - 6m^3}.$$
	As $n \to \infty$ and $m \to \infty$, the expectation of this index converges to $1/3$.
\end{prop}

In an independent work, a degree-based Gini index for general graphs was proposed by~\cite{Domicolo}. This particular Gini index, in general, is used to assess the regularity of classes of random graphs. A follow-up study that uncovers a duality theory is conducted in~\cite{Domicoloduality}. To avoid ambiguity, we denote the degree-based Gini index in ~\cite{Domicolo} by $\tilde{B}_n$.

\begin{prop}[\cite{Domicolo}]
	The mean of the degree-based Gini index for the class of random caterpillars is
	$$\tilde{B}_n \convL \frac{1}{2},$$
	as $n \to \infty$.
\end{prop}

A slightly different degree-based Gini index for random caterpillars 
was discussed in an independent source~\cite{Zhang}, where a 
different target population was considered. It is evident that this 
type of Gini index degenerates as $n$ goes to infinity, shown 
in~\cite{Zhang} via numerical experiments and 
in~\cite{Domicoloduality} via a probabilistic approach.

\section{Hoover index}
\label{Sec:Hoover}

Another important topological measure with applications in economics is the {\em Hoover index}~\cite{Hoover}, which is also known as the Robin Hood Index or the Schutz index in the literature. Like the Gini index, the Hoover index is another inequality metric used for measuring the deviation of the current (income or wealth) distribution from the perfectly even distribution. An alternative interpretation of this index is the portion of population income that would be taken from the richer half to the poorer half for the whole community to be perfectly equal. Mathematically, the {\em Hoover index} of a population with size $n$ is
\begin{equation}
\label{Eq:Hoover}
H=\frac{n}{2}\times \frac{\sum_{i = 1}^{n} |w_i - \bar{w}|}{\bar{w}},
\end{equation}
where $\bar{w}=1/n\sum_{i=1}^n w_i$ is the average of the entire population wealth. A graphical interpretation of the Hoover index is the longest vertical distance between the {\em Lorenz curve} and the $45$ degree line of a unit square. Thus, we immediately have $0 \le H < 1$.

A graph-friendly Hoover index is to replace all $w_i$'s in Equation~(\ref{Eq:Hoover}) with node degrees. Our intent is to consider the Hoover index for a class of graphs. Therefore, we propose a degree-based Hoover index for graphs analogous to the degree-based Gini index introduced in~\cite{Domicolo} as a competing measure for assessing graph regularity. Given an arbitrary graph $G = (V, E) \in \mathsf{G}$, where $V$ and $E$ respectively denotes the vertex set and the edge set of graph $G$ and $\mathsf{G}$ is the class to which $G$ belongs, the degree-based Hoover index is defined as follows:
$$H(G) = \frac{1}{2} \times \frac{\sum_{v \in V} \bigl|{\rm deg}(v) - \sum_{v \in V} {\rm deg}(v) / |V| \big|}{\E \bigl[|\mathcal{V}|\bigr] \times \E \bigl[{\rm deg}(\mathcal{U}) \bigr]},$$
where ${\rm deg}(v)$ represents the {\em degree} (i.e., the number of edges incident to $v$) of~$v$, $|V|$ is the cardinality of set $V$, $\E \bigl[|\mathcal{V}|\bigr]$ is the expected {\em order} (i.e., the number of nodes) of a randomly chosen graph in $\mathsf{G}$, and $\E \bigl[{\rm deg}(\mathcal{U}) \bigr]$ is the degree of a randomly chosen node in a randomly chosen graph in $\mathsf{G}$. The Hoover index of the class $\mathsf{G}$, $H(\mathsf{G})$, is the average of all $H(G)$ for $G$ in $\mathsf{G}$. An argument similar to~\cite[Theorem 4.1]{Domicolo} can be established to show that the Hoover index of a graph class takes values between $0$ and $1$ asymptotically, and a value closer to $0$ suggests that the graphs in the class tend to be more regular.

The order of an arbitrary $C_n$ is fixed, i.e., $(n + m)$. Let $\mathcal{U}_{C_n}$ be a randomly selected node of $C_n$ uniformly chosen from the class of random caterpillars. We have
$$
\E \bigl[{\rm deg}(\mathcal{U}_{C_n}) \bigr] = \frac{\sum_{i = 1}^{m} D_{i, n} + \sum_{i = 1}^{m} X_{i, n}}{n + m} = \frac{2n + 2m - 2}{n + m} = 2 - \frac{2}{n + m}.
$$

Next, we compute the numerator. Note that all leaves have degree $1$, less than the average. The probability that the (spine) nodes at the two ends of the spine are never selected in a long run is negligible. Hence, with high probability, the degree of each spine node is larger than the average. Thus, the expectation of the numerator is equivalent to
$$
\sum_{i = 1}^{m} \left(\E[D_{i, n}] - \frac{2n + 2m - 2}{n + m}\right) + n \, \left(\frac{2n + 2m - 2}{n + m} - 1\right) = \frac{2n(n + m - 2)}{n + m}.
$$
In what follows, we get the asymptotic mean of the Hoover index of the class of random caterpillars at time $n$, denoted by $H_n$, in the next proposition.
\begin{prop}
	\label{Prop:meanH}
	The mean of the Hoover index of the class of random caterpillars 
	$$H_n \, \to \, \frac{1}{2},$$
	as $n \to \infty$.
\end{prop}

\begin{proof}
	By the definition of the Hoover index, we have
	$$
	H_n = \E\bigl[H(C_n)\bigr] = \frac{2n(n + m - 2)/(n + m)}{2(n + m)\bigl(2 - 2/(n + m)\bigr)} \, \to \, \frac{1}{2},
	$$
	as $n \to \infty$.
\end{proof}

\section{Zagreb index}
\label{Sec:Zagreb}

In this section, we calculate the Zagreb index of a random 
caterpillar at time $n$, denoted by $Z_n = \mbox{{\tt 
Zagreb}}(C_n)$. The {\em Zagreb index} of a graph is defined as the 
sum of the squared degrees of the nodes in the graph. Applications 
of Zagreb index mostly appear in mathematical chemistry, used to 
study molecular complexity~\cite{Nikolic}, 
chirality~\cite{Golbraikh}, ZE-isomerism~\cite{Golbraikh2002} and 
heterosystems~\cite{Milicevic}. It is not even possible to list them 
all, so we refer the interested readers to a survey 
article~\cite{Nikolic2003}, in which the authors emphasized the 
potential applicability of Zagreb index for deriving multilinear 
regression models. In the literature of random graphs and 
algorithms, the Zagreb indices of random recursive trees, random 
$b$-ary tree, plain-oriented recursive trees and preferential 
attachment caterpillars were respectively studied in~\cite{Feng, 
Feng2015, Zhangarxiv, Zhangarxiv2019}. 

By definition, the Zagreb index of random caterpillars is given by
$$Z_n = \sum_{i = 1}^{m} D^2_{i, n} + \sum_{i = 1}^{m} X_{i, n}.$$
Here, first, we give some additional useful notations. Let $\indicator_n(i)$ denote the event that the spine node labeled with $i$ is selected at time $n$, and let $\field_n$ denote the $\sigma$-field generated by the history of the growth of a caterpillar in the first $n$ stages. We present the mean of $Z_n$ as well as a weak law in the next proposition.

\begin{prop}
	\label{Prop:meanZ}
	For $n \ge 0$, we have
	$$\E[Z_n] = \frac{n^2}{m} + \frac{(6m - 5)n}{m} + 4m - 6.$$
	As $n \to \infty$, we have
	$$\frac{Z_n}{n^2} \convL \frac{1}{m}.$$
	This convergence takes place in probability as well.
\end{prop}

\begin{proof}
	We consider the following almost-sure recursive relation between $Z_{n - 1}$ and $Z_n$, conditional on $\indicator_n(i)$ and $\field_{n - 1}$:
	\begin{equation}
	\label{Eq:recZ}
	Z_n = Z_{n - 1} + (D_{i, n - 1} + 1)^2 - D^2_{i, n - 1} + 1 = Z_{n - 1} + 2 D_{i, n - 1} + 2.
	\end{equation}
	Averaging it out over $i$, we get
	\begin{equation*}
	\E[Z_n \given \field_{n - 1}] = Z_{n - 1} + \frac{2}{m} \sum_{i = 1}^{m} D_{i, n - 1} + 2 = Z_{n - 1} + \frac{2(n + 2m - 3)}{m} + 2.
	\end{equation*}
	Taking another expectation with respect to $\mathbb{F}_{n - 1}$, we obtain the recurrence for $Z_n$. We solve it with initial condition $\E[Z_0] = Z_0 = 4m - 6$, and get the result stated in the proposition.
	
	As $n \to \infty$, we have $Z_n / n^2$ converges to $1/m$ in $L_1$-space, suggesting that $Z_n / n^2$ converges to $1/m$ in probability as well.
\end{proof}

We continue to calculate the second moment of $Z_n$, and accordingly obtain the variance of $Z_n$.

\begin{prop}
	\label{Prop:varZ}
	For $n \ge 0$, we have
	\begin{align*}
	\E\left[Z^2_n\right] &= \frac{n^4}{m^2} + \frac{(12m - 10)n^3}{m^2} + \frac{(44m^2 - 70m + 23)n^2}{m^2}
	\\ &\qquad{}+ \frac{(48m^3 - 112m^2 + 66m - 14)n}{m^2} + 16\left(m - \frac{3}{2}\right)^2,
	\end{align*}
	and 
	$$\Var[Z_n] = \frac{2n \bigl((m - 1)n + 3m - 7\bigr)}{m^2}.$$
\end{prop}

\begin{proof}
	Recall the almost-sure recursive relation between $Z_{n - 1}$ and $Z_n$, conditional on $\indicator_n(i)$ and $\field_n$, established in Proposition~\ref{Prop:meanZ} (c.f.\ Equation~(\ref{Eq:recZ})). We square it on both sides and get
	\begin{align*}
	Z_n^2 &= (Z_{n - 1} + 2 D_{i, n - 1} + 2)^2
	\\ &= Z^2_{n - 1} + 4 D^2_{i, n - 1} + 4 + 4 Z_{n - 1} + 4 Z_{n - 1}D_{n - 1} + 8 D_{i, n - 1}
	\end{align*}
	Again, we average it out over $i$ to obtain
	\begin{align*}
	\E\left[Z^2_n \given \field_{n - 1}\right] &= Z^2_{n - 1} + \frac{4}{m} \sum_{i = 1}^{m} D^2_{i, n - 1} + 4 + 4 Z_{n - 1} + \frac{4 Z_{n - 1}}{m} \sum_{i = 1}^{m} D_{i, n - 1} 
	\\ &\qquad{}+ \frac{8}{m} \sum_{i = 1}^{m} D_{i, n - 1}.
	\\ &= Z^2_{n - 1} + \frac{4}{m} \bigl(Z_{n - 1} - (n - 1)\bigr) + 4 + 4Z_{n - 1} 
	\\ &\qquad{}+ \frac{4(n + 2m - 3)}{m} Z_{n - 1} + \frac{8(n +2 m - 3)}{m}.
	\end{align*}
	We then obtain the recurrence for the second moment of $Z_n$ by taking the expectation with respect to $\field_{n - 1}$ and by plugging the result of the first moment of $Z_{n - 1}$. Solve the recurrence with initial condition $\E\left[Z^2_0\right] = Z^2_0 = (4m - 6)^2$ to get the result stated in the proposition. In what follows, we obtained the exact expression of the variance of $Z_n$ by subtracting the square of the mean of $Z_n$ from its second moment.
\end{proof}

According to the expression of $\E\left[Z_n^2\right]$, we can simply conclude that 
$$ \frac{Z^2_n}{n^4} \convL \frac{1}{m^2}, $$
done in a similar manner as the proof for $L_1$ convergence of $Z_n/n^2$. Thus, we also obtain a weak law for $Z^2_n$. Besides, we find a stronger $L_2$ convergence of $Z_n/n^2$, presented in the next corollary.

\begin{corollary}
	As $n \to \infty$, we have
	$$ \frac{Z_n}{n^2} \convLL \frac{1}{m}.$$
\end{corollary}

\begin{proof}
	By the $L_1$-convergence results for $Z^2_n$ and $Z_n$, we have
	\begin{align*}
	\lim_{n \to \infty}\E\left[\left|\frac{Z_n}{n^2} - \frac{1}{m}\right|^2\right] &= \lim_{n \to \infty}\E\left[\frac{Z^2}{n^4} - 2 \times \frac{Z_n}{n^2}\times\frac{1}{m} + \frac{1}{m^2}\right]
	\\ &=\frac{1}{m^2} - 2 \times \frac{1}{m} \times \frac{1}{m} + \frac{1}{m^2}
	\\ &= 0.
	\end{align*}
\end{proof}

According to the variance of $Z_n$ given in Proposition~\ref{Prop:varZ}, we find that the order of its leading term is $n^2$, which is the same as that for the mean of $Z_n$. A sharp concentration of the variance often suggests asymptotic normality of the random variable. In what follows, we characterize the asymptotic behavior of $Z_n$ after properly scaled. Based on our investigation, the scale is $n$. Our strategy is to construct a {\em martingale} array based on a transformation of $Z_n$, and appeal to a {\em Martingale Central Limit Theorem} (MCLTs) for developing a Gaussian law of $Z_n/n$ as $n$ goes to infinity.

In the proof of Proposition~\ref{Prop:meanZ}, we find that 
$$\E[Z_n \given \field_{n - 1}] = Z_{n - 1} + \frac{2(n + 2m - 3)}{m} + 2,$$
suggesting that $\{Z_n\}_n$ is not a martingale. In the next lemma, we apply a transformation to $Z_n$ such that the new sequence is a martingale.

\begin{lemma}
	\label{Lem:martingale}
	The sequence of $\{M_n\}_n$ such that 
	$$M_n = Z_n - \frac{n(n + 6m - 5)}{m}$$
	is a martingale.
\end{lemma}
\begin{proof}
	Let us consider $M_n = Z_n + \beta_{n}$ such that the following fundamental property of martingale holds; namely,
	\begin{align*}
	\E[Z_n + \beta_n \given \field_{n - 1}] = Z_{n - 1} + \frac{2(n + 2m - 3)}{m} + 2 + \beta_n = Z_{n - 1} + \beta_{n - 1}.
	\end{align*}
	This produces a recurrence for $\beta_n$, namely,
	$$\beta_n = \beta_{n - 1} - \frac{2n + 6m - 6}{m}.$$
	The solution is $\beta_n = -n(n + 6m - 5)/m$, obtained by taking an arbitrary choice of the initial condition; we choose $\beta_0 = 0$.
\end{proof}

The MCLT that we exploit to show asymptotic normality is based {\color{red} on} martingale differences, expressed in terms of a difference operator, i.e., $\nabla M_j := M_j - M_{j - 1}$. In fact, there are different versions of MCLT listed in~\cite{Hall}, requiring different sets of conditions. The MCLT that we use refers to~\cite[Corollary 3.2]{Hall}, requiring two conditions which are respectively known as the {\em conditional Lindeberg's condition} and the {\em conditional variance condition}. In the next two lemmas, we verify these two conditions one after another.

\begin{lemma}
	\label{Lem:Lindeberg}
	The Lindeberg's condition is given by
	$$U_n := \sum_{j = 1}^{n} \E \left[ \left( \frac{\nabla M_j}{n} \right)^2 \indicator \bigl(| \nabla M_j / n| > \varepsilon\bigr) \Given \field_{j - 1}\right] \convP 0,$$
	for arbitrary $\varepsilon > 0$.
\end{lemma}

\begin{proof}
	We establish an absolutely uniform bound in all $j \le n$. By the construction of the martingale, we have
	\begin{align*}
	|\nabla M_j| &= |M_j - M_{j - 1}|
	\\ &\le |Z_j - Z_{j - 1}| + \frac{2j + 6m - 6}{m}
	\\ &\le 2 \max_{i} D_{i, j - 1} + 2 + \frac{2j + 6m - 6}{m}
	\\ &\le 2\bigl(2 + (j - 1)\bigr) + 2 + \frac{2j + 6m - 6}{m}
	\\ &= \left(\frac{2m + 2}{m}\right) j + \frac{10m - 6}{m},
	\end{align*}
	which is increasing in $j$ for any fixed integer $m \ge 2$. Thus, we conclude that $|\nabla M_j|$ is uniformly bounded by $n$. Hence, for any $\varepsilon > 0$, there exists $n_0(\varepsilon) > 0$ such that the sets $\{|M_j / n| > \varepsilon\}$ are all empty for $n > n_0(\varepsilon)$. In what follows, we conclude that $U_n$ converges to $0$ almost surely, which is stronger than the  in-probability convergence required for the Lindeberg's condition.
\end{proof}

\begin{lemma}
	\label{Lem:variance}
	The conditional variance condition is given by
	$$V_n := \sum_{j = 1}^{n} \E \left[ \left( \frac{\nabla M_j}{n} \right)^2 \Given \field_{j - 1}\right] \convP \frac{2(m - 1)}{m^2}.$$
\end{lemma}

\begin{proof}
	We rewrite $V_n$ as follows:
	\begin{align*}
	V_n &= \frac{1}{n^2} \sum_{j = 1}^{n} \E \left[\bigl(Z_j + \beta_j - (Z_{j - 1} + \beta_{j - 1})\bigr)^2 \ggiven \field_{j - 1}\right]
	\\ &= \frac{1}{n^2} \sum_{j = 1}^{n} \E \bigl[(Z_j - Z_{j - 1})^2 + 2(Z_j - Z_{j - 1})(\beta_j - \beta_{j - 1}) 
	\\ &\qquad{}+ (\beta_j - \beta_{j - 1})^2 \given \field_{j - 1}\bigr]
	\end{align*}
	We evaluate the three parts in the summand one after another, by considering the asymptotic equivalents of $Z^2_{j - 1}$ and $Z_{j - 1}$.
	\begin{enumerate}
		\item The first part is
		\begin{align*}
		\E[(Z_j - Z_{j - 1})^2 \given \field_{j - 1}] &= \E[Z_j^2 \given \field_{j - 1}] - 2 Z_{j - 1} \E[Z_j \given \field_{j - 1}] + Z^2_{j - 1}
		\\&= \frac{4j^2 + 28(m - 1)j + 6m^2 - 68m + 24}{m^2}.
		\end{align*}
		\item The second part is
		\begin{align*}
		\E[2(Z_j - Z_{j - 1})(\beta_j - \beta_{j - 1}) \given \field_{j - 1}] &= 2(\beta_j - \beta_{j - 1}) \bigl(\E[Z_j \given \field_{j - 1}] - Z_{j - 1}\bigr)
		\\ &= -\frac{8(j + 3m + 3)(j + 3m - 3)}{m^2}.
		\end{align*}
		\item The third part it
		$$\E[(\beta_j - \beta_{j - 1})^2 \given \field_{j - 1}] = (\beta_j - \beta_{j - 1})^2 = \frac{(2j + 6m + 6)^2}{m^2}.$$
	\end{enumerate}
	Putting three parts together, we obtain the summand in $V_n$ for each $j$. Then, we sum these terms for $j = 1, 2, \ldots, n$, and let $n$ go to infinity, obtaining
	$$
	V_n = \frac{1}{n^2} \left(\frac{2(m - 1)}{m^2} n^2 + O(n)\right) \convL \frac{2(m - 1)}{m^2}.
	$$ 
	The $L_1$ convergence is stronger than the required  in-probability convergence.
\end{proof}

\begin{theorem}
	\label{Thm:zagreb}
	As $n \to \infty$, $Z_n / n$ follows a Gaussian law, namely,
	$$\frac{Z_n - n^2/m}{n} \convD \mathcal{N}\left(0, \frac{2(m - 1)}{m^2}\right).$$
\end{theorem}
The proof is easily verified by using the MCLT~\cite[Corollary 3.2]{Hall}. 

\section{Randi\'{c} index}
\label{Sec:Randic}

In this section, the Randi\'{c} index of a random caterpillar, denoted by $R_n = \mbox{{\tt Randi\'{c}}}(C_n)$, is calculated. The {\em Randi\'{c} index} of a graph $G = (V, E)$ (with parameter $\alpha$) is the sum of the product of the degrees (raised to power $\alpha$) of all pairwise connected nodes. Mathematically, it is
$$R(G) = \sum_{\{u, v\} \in E} \bigl({\rm deg}(u){\rm deg}(v)\bigr)^{\alpha},$$
for all $u, v \in V$. 

The classical choice of $\alpha$ is $-1/2$~\cite{Randic}. Under such choice, the Randi\'{c} index is also called {\em connectivity index}. The general definition given above was later proposed by~\cite{Bollobas1998}. Similar to Zagreb index, Randi\'{c} is also used for modeling QSAR and QSPR of chemical compounds (e.g., alkanes~\cite{Gutman1999}, saturated hydrocarbon~\cite{Gutman2000} and benzenoid systems~\cite{Rada}) in chemoinformatics.  We refer interested readers to~\cite{Randic2001} for a concise survey, to~\cite{Randic2008} for a complete history review and to~\cite{Li} for a summary of mathematical properties. To the best of our knowledge, little work has been done for the Randi\'{c} index of random graph models. The only source that we find in the literature is the Randi\'{c} index of random binary trees~\cite{Feng2008}.

In $C_n$, note that no pair of leaves is connected. Each leaf, instead, is only connected to its parent spine node. Therefore, the contribution by each edge connecting a leaf and its corresponding spine node to $R_n$ is the degree of the spine node raised to power $\alpha$. There are $(m - 1)$ edges on the spine, and the contribution by each of them to $R_n$ is $(D_{i - 1, n} D_{i, n})^{\alpha}$ for $i = 2, 3, \ldots, m$. Hence, we arrive at
\begin{equation}
\label{Eq:Randic}
R_n = \sum_{i = 2}^{m} (D_{i - 1, n} D_{i, n})^{\alpha} + \sum_{i = 1}^{m} X_{i, n} D^{\alpha}_{i, n}.
\end{equation}

Specifically, we consider the Randi\'{c} index with parameter $\alpha = 1$. This particular type of Randi\'{c} index is also popular in mathematical chemistry, viewed as a molecular structure descriptor~\cite{Balaban}. Among all, this index is best known for measuring the branching of molecular carbon-atom skeleton~\cite{Gutman, Gutman2013}. In the rest of this section, all the Randi\'{c} indicies without specification all refer to the Randi\'{c} indicies with parameter $\alpha = 1$. Accordingly, we redefine $R_n$ in Equation~(\ref{Eq:Randic}) as
$$R_n = \sum_{i = 2}^{m} D_{i - 1, n} D_{i, n} + \sum_{i = 1}^{m} X_{i, n} D_{i, n}.$$

Many mathematical properties of this specific class of Randi\'{c} indices ($\alpha = 1$) were investigated by graph theorists and combinatorists recently~\cite{Gutman2015, Khalifeh, Mehdi}. In the graph theory community, the Randi\'{c} index is more often known by a different name: the {\em second-order Zagreb index}. In~\cite{Bollobas}, the Randi\'{c} index for extremal graphs was studied, where another different name ``extreme $1$-weight " was used. In the next proposition, we present the mean of $R_n$ as well as a weak law.

\begin{prop}
	\label{Prop:meanR}
	For $n \ge 0$, the mean of the Randi\'{c} index of a caterpillar is
	$$\E[R_n] = \frac{(2m - 1)n^2 + (7m^2 - 10m + 1)n + 4m^2(m - 2)}{m^2}.$$
	As $n \to \infty$, we have
	$$\frac{R_n}{n^2} \convL \frac{2m - 1}{m^2}.$$
	This convergence takes place in probability as well.
\end{prop}

\begin{proof}
	Recall that the the distribution of $X_{i, n}$'s is multinomial with parameters $n$ and $\mathbb{I}_{m \times 1}/m$. Hence, we have $\E[X_{i, n}] = n/m$, $\Var[X_{i, n}] = (m - 1)n/m^2$, and $\E\left[X^2_{i, n}\right] = (n + m - 1)n/m^2$ for each $i$, and $\E[X_{i, n}X_{j, n}] = -n/m^2 + n^2/m^2 = n(n - 1)/m^2$ for all $i \neq j$. In what follows, we obtain
	\begin{align*}
	\E[R_n] &= \sum_{i = 2}^{m} \E[D_{i - 1, n} D_{i, n}] + \sum_{i = 1}^{m} \E[X_{i, n} D_{i, n}]
	\\ &= 2 \, \E\bigl[(X_{1, n} + 1)(X_{2, n} + 2)\bigr] + (m - 3) \, \E\bigl[(X_{2, n} + 2)(X_{3, n} + 2)\bigr] 
	\\&\qquad{}+ 2 \, \E\bigl[X_{1, n} (X_{1, n} + 1)\bigr] + (m - 2) \, \E\bigl[X_{2, n} (X_{2, n} + 2)\bigr]
	\\&= (m - 1) \, \E[X_{1, n}X_{2, n}] + m \, \E\left[X^2_{1, n}\right] + (6m - 8) \, \E[X_{1, n}] + (4m - 8)
	\\&=\frac{(m - 1)n(n - 1)}{m^2} + \frac{m(n + m - 1)n}{m^2} + \frac{(6m - 8)n}{m} + 4m - 8
	\\& = \frac{(2m - 1)n^2 + (7m^2 - 10m + 1)n + 4m^2(m - 2)}{m^2}.
	\end{align*}
	
	It is obvious that $R_n / n^2$ converges to $(2m - 1)/m^2$ in $L_1$-space, and this convergence is stronger than the in-probability convergence stated in the weak law.
\end{proof}

Another approach to calculating the mean of $R_n$ is to establish an almost-sure relation analogous to Equation~(\ref{Eq:recZ}) for $R_n$, to obtain an recurrence for $\E[R_n]$ by taking expectation twice, and lastly to solve the recurrence. We omit the details of the derivation, but just present an intermediate step that we need for the follow-up study:
\begin{align*}
\E[R_j \given \field_{j - 1}] &= R_{j - 1} + \frac{1}{m}\bigl(2(2j + 3m - 5) - D_{1, j - 1} - D_{m, j - 1}\bigr) + 1
\\ &\ge R_{j - 1} + \frac{2j + 7m - 10}{m},
\end{align*}
suggesting that $\{R_j\}_{1 \le j \le n}$ is a super-martingale. Note that we use $j$ as subscript instead of $n$ to avoid potential confusion of notation. Consider a generic scale, $\xi_n = n^2$, free of $j$. It is obvious that $\{R_j / \xi_n \}_{1 \le j \le n}$ remains a super-martingale.

\begin{theorem}
	\label{Thm:almostsureR}
	There exists a random variable $\tilde{R}$ finite in its mean, such that
	$$\frac{R_n}{n^2} \convas \tilde{R},$$
	as $n \to \infty$.
\end{theorem}

\begin{proof}
	Notice that $R_j$ is increasing in $j$. So is $R_n / \xi_n$. We thus have
	\begin{equation*}
	\sup_j \E \left[\left|\frac{R_j}{n^2}\right|\right] = \E\left[\frac{R_n}{n^2}\right] = \frac{2m - 1}{m^2} + O\left(\frac{1}{n}\right) < + \infty.
	\end{equation*}
	We thus arrive at the conclusion by the {\em Doob's Convergence Theorem}.
\end{proof}

\section{Wiener index}
\label{Sec:Wiener}

In this section, we place focus on the Wiener index of a random caterpillar, i.e., $W_n = \mbox{{\tt Wiener}}(C_n)$. The Wiener index of a graph is defined as the sum of the lengths of the shortest paths (i.e., distances) of all pairs of nodes therein. Mathematically, given $G = (V, E)$, it is
$$W(G) = \sum_{u, v \in V} {\rm dist}(u, v),$$
where ${\rm dist}(u, v)$ is the distance between $u$ and $v$.

In chemistry, the Wiener index was first used to study carbon-carbon bonds between all pairs of carbon atoms in an alkane~\cite{Wiener}, and later was used to characterize QSPR for alkanes~\cite{Platt}. There is a variety of applications of the Wiener index in chemical graph theory and chemoinformatics, not limited to QSAR and QSPR modeling. For the sake of conciseness, we refer the interested readers to~\cite{Nikolic1995} and relevant references therein.

On the other hand, the Wiener index is very popular in the random graph community, and is extensively investigated for many random models, such as random binary search trees and random recursive trees~\cite{Neininger}, balanced binary trees~\cite{Bereg}, random digital trees~\cite{Fuchs}, random split trees~\cite{Munsonius2011}, random $b$-ary trees~\cite{Munsonius}, conditioned Galton-Watson trees~\cite{Fill} and more general rooted and unrooted trees~\cite{Wagner} and simply generated random trees~\cite{Janson}.

The Wiener index of a caterpillar is constituted by three classes of contributions:
\begin{enumerate}
	\item The total distances between spine nodes and spine nodes;
	\item The total distances between leaves and leaves;
	\item The total distances between spine nodes and leaves.
\end{enumerate}
We calculate the three contributions to $W_n$ one after another.

The contribution purely among spine nodes is simple, as it is fixed. That is
$$\sum_{i = 1}^{m - 1} \sum_{j = i + 1}^{m} (j - i) =  \frac{1}{2} \sum_{i = 1}^{m - 1} (m + 1 - i)(m - i) = \frac{1}{6}m(m^2 - 1).$$

To compute the contributions between leaves and leaves, we consider the following two scenarios. For all $1 \le i < j \le m$, the distance between a leaf attached to spine node $i$ and a leaf attached to spine node $j$ is $(j - i + 2)$, and the number of pairs is $X_{i, n}X_{j, n}$. On the other hand, for each $1 \le i \le m$, the distance between a leaf attached to spine node $i$ and another (different) leaf attached to the same spine node is $2$, and the number of pairs is $X_i(X_i - 1)/2$. Thus, the total contribution among leaves is
$$
\sum_{i = 1}^{m - 1} \sum_{j = i + 1}^{m} (j - i + 2) X_{i, n}X_{j, n} + \sum_{i = 1}^{m} X_{i, n}(X_{i, n} - 1).
$$

Lastly, we calculate the contributions between leaves and spine nodes. Note that the computation for this class is slightly different from the previous two. We consider the orders of indices $i$ and $j$ to avoid double counting, but it is unnecessary here. For all $1 \le i, j \le m$, the distance between a leaf attached to spine node $i$ and spine node $j$ is $(|i - j| + 1)$, and the number of leaves attached to spine node $i$ is $X_{i, n}$. Hence, the total contribution by this class is given by
$$
\sum_{i = 1}^{m} \sum_{j = 1}^{m} (|j - i| + 1) X_{i, n}.
$$

We are now ready to calculate the mean of the Wiener index. Likewise, we obtain a weak law.

\begin{prop}
	\label{Prop:meanW}
	For $n \ge 0$, the mean of the Wiener index of a caterpillar is
	$$\E[W_n] = \frac{(m^2 +6m - 1)n^2 + (m - 1)(2m^2 + 7m - 1)n + m^2(m^2 - 1)}{6m}.$$
	As $n \to \infty$, we have
	$$\frac{W_n}{n^2} \convL \frac{m^2 + 6m - 1}{6m}.$$
	The convergence takes place in probability as well.
\end{prop}

\begin{proof}
	Putting the three types of contributions that we have calculated together, and taking the expectation for all $X_{i, n}$'s, we get
	\begin{align*}
	\E[W_n] &= \frac{m(m^2 - 1)}{6} + \frac{m(m + 7)(m - 1)}{6} \, \E[X_{1, n}X_{2, n}] + m \, \E\left[X^2_{1, n}\right]
	\\ &\qquad{}- m \, \E\left[X_{1, n}\right] + \left(m^2 + 2 \times {m + 1 \choose 3} \right) \E[X_{1, n}]
	\\ &= \frac{m(m^2 - 1)}{6} + \frac{(m + 7)(m - 1)n(n - 1)}{6m} + \frac{(n + m - 1)n}{m}
	\\ &\qquad{} + \frac{(m + 4)(m - 1)n}{3}
	\\ &= \frac{(m^2 +6m - 1)n^2 + (m - 1)(2m^2 + 7m - 1)n + m^2(m^2 - 1)}{6m}.
	\end{align*}
	The weak law stated in the proposition follows immediately.
\end{proof}

There exists a finite-mean random variable, to which $W_n$ converges, after properly scaled (by $n^2$). The proof is done mutandis mutatis, with an application of the Doob's Convergence Theorem. We thus omit the details, but only state the theorem.

\begin{theorem}
	There exists a random variable $\tilde{W}$ finite in its mean, such that
	$$\frac{W_n}{n^2} \convas \tilde{W},$$
	as $n \to \infty$.
\end{theorem}

An extension of the Wiener index is the {\em hyper Wiener index}, a relatively new topological index that characterizes molecular structure and feature of more complex chemical compounds. This index was proposed by~\cite{Randic1993} to analyze the structure of 2-Methylhexane, and later on was used in~\cite{Darafsheh} to investigate one-pentagonal carbon nanocone. The hyper Wiener index of $G = (V, E)$ is given by
$$W^{h}(G) = \sum_{u, v \in V} \bigl({\rm dist}(u, v) + {\rm dist}^2(u, v) \bigr).$$

Let $W^h_n$ be the hyper Wiener index of $C_n$. The computation of $W^h_n$ is indeed similar to that of $W_n$. We thus only list the key steps. We again decompose the index into three parts, where the first refers to spine-spine contributions:
$$\sum_{i = 1}^{m - 1} \sum_{j = i + 1}^{m} \bigl((j - i) + (j - i)^2\bigr) = \frac{m(m^3 + 2m^2 - m - 2)}{12},$$
which is deterministic. The second part is leaf-leaf contributions, the expectation of which is
\begin{align*}
&\frac{(m^2 + 11m + 46)(m - 1)m}{12} \, \E[X_{1, n}X_{2, n}] + 3m \left(\E\left[X^2_{1, n}\right] - \E[X_{1, n}]\right)
\\ &= \frac{(m^3 + 10m^2 +35m - 10)n^2 - (m + 10)(m + 1)(m - 1)n}{12m}.
\end{align*}
The last part is contributed by the distances between spine nodes and leaf nodes. Its expectation is
$$\frac{(m^2 + 7m + 18)(m - 1)m}{6} \, \E[X_{i, n}] = \frac{(m^2 + 7m + 18)(m - 1)n}{6}.$$

Putting three parts together, we get the expectation of the hyper Wiener index of $C_n$, presented in the next proposition.

\begin{prop}
	For $n \ge 0$, the mean of the hyper Wiener index of a caterpillar is
	\begin{align*}
	\E\left[W^h_n\right] &=  \frac{1}{12m} \bigl((m^3 + 10 m^2 + 35m - 10)n^2 + (2m^3 + 13m^2 + 25m - 10)
	\\ &\qquad{}\times (m - 1) n + m^2(m + 2)(m + 1)(m - 1)\bigr).
	\end{align*}
	As $n \to \infty$, we have
	$$\frac{W^h_n}{n^2} \convL \frac{m^3 + 10 m^2 + 35m - 10}{12m}.$$
	The convergence takes place in probability as well.
\end{prop}

\section{Numerical experiments}
\label{Sec:numeric}

We conduct a series of numerical experiments to verify the results 
developed from Section~\ref{Sec:Hoover} to~\ref{Sec:Wiener}. Given a 
fixed $m =200$, we independently generate $R = 500$ replications of 
random caterpillars after $n = 5000$ evolutionary steps. For each 
simulated caterpillar, its Hoover, Zagreb, Randi\'{c}, Wiener and 
hyper Wiener indices are computed and recorded. We evaluate each of 
the simulation results one after another.

Our results from Section~\ref{Sec:Hoover} indicate that the Hoover 
index of a random caterpillar is a deterministic function of $m$ and 
$n$, and hence lacks randomness. From the simulation, we find that 
the Hoover index of each generated caterpillar is around $0.4807$, 
which is consistent with Proposition~\ref{Prop:meanH}.

Next, we compute the Zagreb index of each generated caterpillar, and 
standardize the (random) sample of Zagreb indices according to 
Propositions~\ref{Prop:meanZ} and~\ref{Prop:varZ}. The resulting 
histogram is depicted in Figure~\ref{Fig:zagreb}. In addition, we 
use the kernel method to estimate the density based on the sample of 
the standardized Zagreb indices, presented in 
Figure~\ref{Fig:zagreb} as well. The simulation results suggest that 
the Zagreb index (after 
proper scaling) of random caterpillars follows a Gaussian law, which 
agrees with Theorem~\ref{Thm:zagreb}. We further confirm the 
conclusion via the {\em Shapiro-Wilk} normality test, which yields 
that the $p$-value equals $0.5442$.
\begin{figure}[tbp]
	\begin{center}
		\includegraphics[width = 0.75\textwidth]{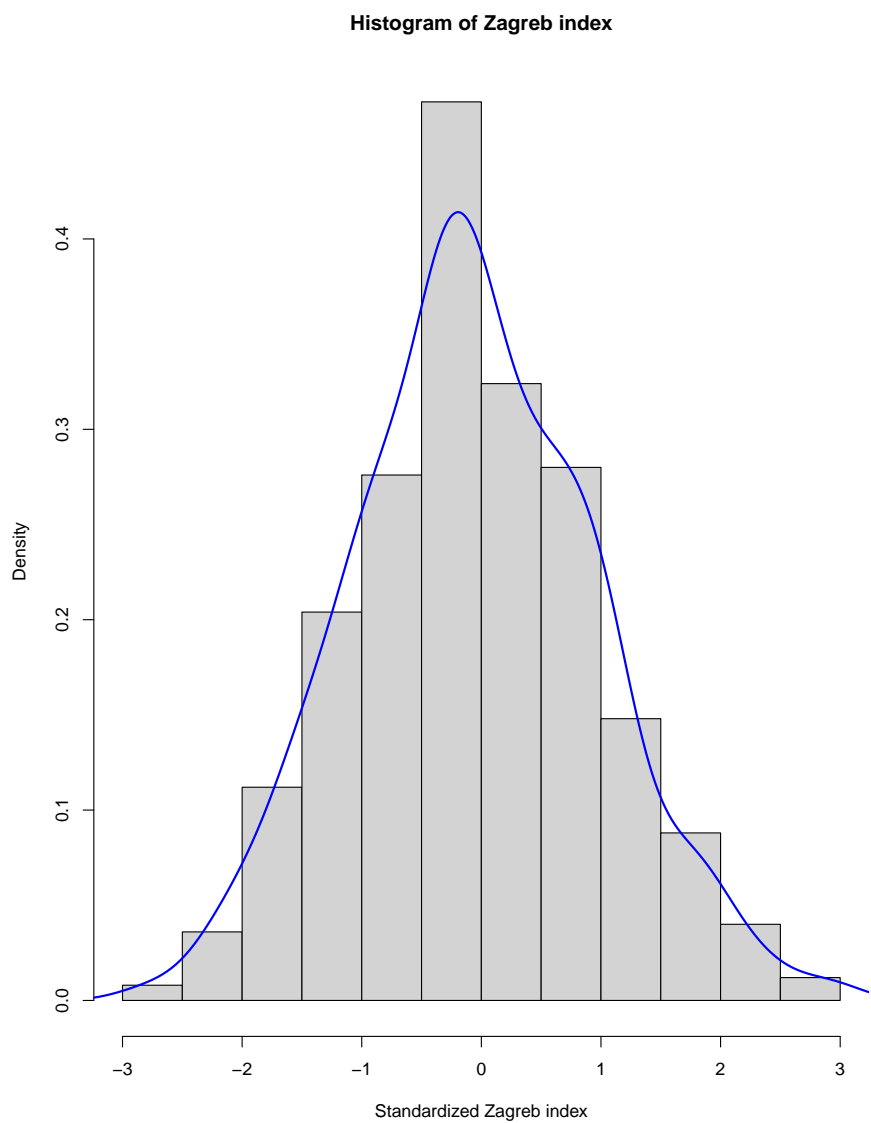}
	\end{center}
	\caption{Histogram of the standardized Zagreb indices of $500$ 
	independently generated random caterpillars with $m = 200$ and 
	$n = 5000$; the thick blue curve is the estimated density of the 
	sample.}
	\label{Fig:zagreb}
\end{figure}

For the Randi\'{c} index, the sample mean based on the simulated 
caterpillars is $0.012$ after scaled by $n^2$. This simulation 
result is reasonably close to the theoretical result $0.010$ from 
Proposition~\ref{Prop:meanR}. 

For the Wiener index, we use the {\tt igraph} package (from {\tt R} 
program) to compute the distances among the vertices in the 
generated caterpillars. The distance 
calculation requires a great deal of computation powers, so we 
adjust the 
simulation parameters to $R^{\prime} = 500$, $m^{\prime} = 50$ and 
$n^{\prime} = 2000$. The average of the Wiener indices (after 
properly scaled) of the simulated caterpillars is $9.7732$, almost 
identical to the theoretical result $9.7720$ from 
Proposition~\ref{Prop:meanW}. Under the same setting of 
$R^{\prime}$, $m^{\prime}$ and $n^{\prime}$, we also compute the 
hyper Wiener indices. The simulation and 
theoretical results (after properly scaled) are respectively given 
by $264.7783$ and $264.6214$, which completes the verification.

\section{Concluding remarks}
In this section, we address some concluding remarks, and propose some potential future work as well. We investigate several popular statistical indicies for a class of random caterpillars, including Gini index, Hoover index, Zagreb index, Randi\'{c} index and Wiener index (and its extension, hyper Wiener index). The mean of each index is computed. Specifically, we show that the limit distribution of the Zagreb index of random caterpillars is Gaussian. Topological index of random graphs is a burgeoning research area in the applied probability community. The follow-up study of the present paper may be given to further investigations of the limit distribution of the Wiener index and the analysis of the Randi\'{c} index with a more general parameter $\alpha$. Brand new research in this area is three folded:
\begin{enumerate}
	\item Propose novel indices according to practical needs; for instance, we can consider a global metric that captures the total weight of random caterpillars, where the weight can be added to nodes or edges or itself can be temporal.
	\item Investigate other statistical indices that are not yet covered in the present paper; for instance, the {\em Hosoya's $Z$ index} that counts the number of matchings in a graph and the {\em Balaban's index} that interprets graph connectivity via the associated distance matrix.
	\item Consider other types of random trees or more complex random networks that can be used for modeling molecular structures of chemical compounds, and investigate the relevant topological indices thereof. We will report our results elsewhere.
\end{enumerate}

%\bibliographystyle{amsplain}
%\bibliography{yourbibfilename}

% add below the content of your .bbl file produced by bibtex.

\end{document}